\documentclass[11pt]{amsart}
\usepackage[linkcolor=blue, citecolor = blue]{hyperref}
\hypersetup{colorlinks=true}

\usepackage{amsthm,amsfonts,amsmath,amssymb,amscd} 

\usepackage{verbatim}
\usepackage{float}
\usepackage{wrapfig}
\usepackage{mathtools}
\usepackage[color=orange!20, linecolor=orange]{todonotes}

\usepackage{enumitem}

\usepackage{ytableau}

\usepackage{tikz}
\usetikzlibrary{decorations.markings}
\tikzstyle{vertex}=[circle, draw, inner sep=0pt, minimum size=4pt]

\usetikzlibrary{arrows,automata}
\usepackage{tikz, tikz-3dplot, pgfplots}
\usetikzlibrary[positioning,patterns]

\newtheorem{theorem}{Theorem}[section]
\newtheorem{proposition}[theorem]{Proposition}
\newtheorem{lemma}[theorem]{Lemma}
\newtheorem{corollary}[theorem]{Corollary}
 
\theoremstyle{definition}
\newtheorem{definition}[theorem]{Definition}
\newtheorem{example}[theorem]{Example}

\theoremstyle{remark}
\newtheorem{remark}[theorem]{Remark}

\newcounter{x}
\newcounter{y}
\newcounter{z}

\newcommand\xaxis{210}
\newcommand\yaxis{-30}
\newcommand\zaxis{90}

\newcommand\topside[3]{
  \fill[fill=white, draw=black,shift={(\xaxis:#1)},shift={(\yaxis:#2)},
  shift={(\zaxis:#3)}] (0,0) -- (30:1) -- (0,1) --(150:1)--(0,0);
}

\newcommand\leftside[3]{
  \fill[fill=cyan, draw=black, shift={(\xaxis:#1)},shift={(\yaxis:#2)},
  shift={(\zaxis:#3)}] (0,0) -- (0,-1) -- (210:1) --(150:1)--(0,0);
}

\newcommand\rightside[3]{
  \fill[fill=lightgray, draw=black,shift={(\xaxis:#1)},shift={(\yaxis:#2)},
  shift={(\zaxis:#3)}] (0,0) -- (30:1) -- (-30:1) --(0,-1)--(0,0);
}
\newcommand\rrightside[3]{
  \fill[fill=white, draw=black,shift={(\xaxis:#1)},shift={(\yaxis:#2)},
  shift={(\zaxis:#3)}] (0,0) -- (30:1) -- (-30:1) --(0,-1)--(0,0);
}

\newcommand\cube[3]{
  \topside{#1}{#2}{#3} \leftside{#1}{#2}{#3} \rightside{#1}{#2}{#3}
}
\newcommand\ccube[3]{
  \topside{#1}{#2}{#3} \leftside{#1}{#2}{#3} 
}

\newcommand\planepartition[1]{
 \setcounter{x}{-1}
  \foreach \a in {#1} {
    \addtocounter{x}{1}
    \setcounter{y}{-1}
    \foreach \b in \a {
      \addtocounter{y}{1}
      \setcounter{z}{-1}
      \foreach \c in {0,...,\b} {
        \addtocounter{z}{1}
      \ifthenelse{\c=0}{\setcounter{z}{-1},\addtocounter{y}{0}}{
        \cube{\value{x}}{\value{y}}{\value{z}}}
      }
    }
  }
}

\newcommand\pplanepartition[1]{
 \setcounter{x}{-1}
  \foreach \a in {#1} {
    \addtocounter{x}{1}
    \setcounter{y}{-1}
    \foreach \b in \a {
      \addtocounter{y}{1}
      \setcounter{z}{-1}
      \foreach \c in {0,...,\b} {
        \addtocounter{z}{1}
      \ifthenelse{\c=0}{\setcounter{z}{-1},\addtocounter{y}{0}}{
        \ccube{\value{x}}{\value{y}}{\value{z}}}
      }
    }
  }
}

\title[]{Random plane partitions and corner distributions} 

\author[Damir Yeliussizov]{Damir Yeliussizov}

\address{KBTU, Almaty, Kazakhstan}
\email{\href{mailto:yeldamir@gmail.com}{yeldamir@gmail.com}}

\begin{document}

\begin{abstract}
We explore some probabilistic applications arising in connections with $K$-theoretic symmetric functions.
For instance, we determine certain corner distributions of random lozenge tilings and plane partitions. We also introduce some distributions that are naturally related to the corner growth model. Our main tools are dual symmetric Grothendieck polynomials and normalized Schur functions.
\end{abstract}

\maketitle


\section{Introduction}
Combinatorics arising in connection with $K$-theoretic Schubert calculus is quite rich. Accompanied by certain families of symmetric functions, it usually presents  some inhomogeneous deformations of objects beyond classical Schur (or Schubert) case. While the subject is intensively studied from combinatorial, algebraic and geometric aspects, see \cite{lenart, buch, buch1, vakil, lp, in, dy} and many references therein, 
much less is known about probabilistic connections (unlike interactions between probability and representation theory). Some work in this direction was done in \cite{ty, ms} and related problems were addressed in \cite{dy3}.\footnote{See also Remark~\ref{upd} on recent works.}

\vspace{0.5em}

In this paper, we give several probabilistic applications obtained with tools 
from combinatorial $K$-theory. We mostly focus on one 
deformation of Schur functions, the {\it dual Grothendieck polynomials}, whose associated combinatorics is fairly neat.   
In particular, we show that these functions are naturally related 
to the {\it corner growth model} \cite{joh1, joh2, sep, romik} (which can also be viewed as a totally asymmetric simple exclusion process or a directed last passage percolation), see subsec.~\ref{corner0} and Sec.~\ref{corner}. 

\vspace{0.5em}

We are now going to discuss our results across few related models.

\subsection{$\mathbb{N}$-matrices with bounded last passage time} 
A lattice path $\Pi$ with vertices indexed by $\mathbb{N}^2$ (where $\mathbb{N} = \{0,1,2,\ldots \}$) is called a {\it monotone path} 
if it uses only steps of the form $(i,j) \to (i+1, j), (i, j+1)$.
An {\it $\mathbb{N}$-matrix} is a matrix of nonnegative integers with only finitely many nonzero entries. Let $a,b,c$ be positive integers. Given  an $\mathbb{N}$-matrix $D = (d_{ij})$ with $b$ rows and $c$ columns, the {\it last passage time} $G(b,c)$ is defined as 
$$
G(b,c) = \max_{\Pi}\sum_{(i,j) \in \Pi} d_{i j},
$$
where the maximum is taken over {monotone paths}  
$\Pi$ from $(1,1)$ to $(b,c)$. 
Let $\mathrm{BM}(a,b,c)$ be the set of $b \times c$ $\mathbb{N}$-matrices whose 
last passage time is 
{\it bounded} by $a$, i.e.
\begin{align*}
\mathrm{BM}(a,b,c) := \left\{D = (d_{i j})_{i,j = 1}^{b,c} : 
\quad G(b,c) \le a,
\quad d_{ij} \in \mathbb{N} 
\right\}
\end{align*}
This set is in fact equinumerous with {\it boxed plane partitions}  (see Theorem~\ref{bij}). 

Consider the uniform probability measure on $\mathrm{BM}(a,b,c)$. For a random matrix $D \in \mathrm{BM}(a,b,c)$, define the {\it column marginals} 
$$C_{\ell} = \sum_{i = 1}^b d_{i \ell}$$ 
i.e. as the sum of entries in $\ell$-th column of $D$ where $\ell \in [1,c]$.
It is natural to ask what are distributions of $C_{\ell}$. 
We determine these distributions. The results about $C_{\ell}$ are presented via lozenge tilings (see Theorem~\ref{gam}) below. We show that the random variables $(C_{\ell})$ are exchangeable (see Theorem~\ref{sym1}; in particular, $C_{i}$ and $C_{j}$ have the same distribution for all $i, j \in [1,c]$) and obtain limiting (joint) distributions with different asymptotic regimes for the parameters $a,b,c$. In some regimes, the   variables $C_{\ell}$ become {\it asymptotically independent}. By symmetry, we also get similar results for {\it row marginals}.

\subsection{Boxed plane partitions}
A {\it plane partition} is an $\mathbb{N}$-matrix $\pi = (\pi_{i j})_{i,j\ge 1}$  
satisfying 
$$
\pi_{i j} \ge \pi_{i+1, j}, \qquad \pi_{i j} \ge \pi_{i, j+1}, \qquad i,j \ge 1.
$$
By default we ignore zeros of $\pi$. We denote by $\mathrm{sh}(\pi) := \{(i,j) : \pi_{ij} > 0 \}$ the {\it shape} of $\pi$. 

Let $\mathrm{PP}(a,b,c)$ be the set of {\it boxed} plane partitions that fit inside the box $a \times b \times c$, i.e. the first row of the shape is at most $a$, the first column is at most $b$, and the first entry is at most $c$. 
We show that $|\mathrm{PP}(a,b,c)| = |\mathrm{BM}(a,b,c)|$ (Theorem~\ref{bij}).

Consider the uniform probability measure on the set $\mathrm{PP}(a,b,c)$. Recall that classical {MacMahon's theorem} on boxed plane partitions gives the explicit product formula 
\begin{align*}
Z_{a b c} := |\mathrm{PP}(a,b,c)| = \prod_{i = 1}^{a} \prod_{j = 1}^{b} \prod_{k = 1}^{c} \frac{i + j + k - 1}{i + j + k - 2} 
\end{align*}
Note that we also have 
\begin{align*}
Z_{a b c} = s_{(a^b)}(1^{b+c}),
\end{align*}
where $s_{\lambda}$ is the {\it Schur polynomial} and $1^n := (1, \ldots, 1)$ repeated $n$ times. (See e.g. \cite[Ch.~7]{sta}.) 

For a uniformly random plane partition $\pi \in \mathrm{PP}(a,b,c)$, let 
$X_{\ell}$ be the number of 
columns of $\pi$ that contain entry $\ell \in [1,c]$.  
In fact, the variables $X_{\ell}$ become image of the marginals~$C_{\ell}$ described previously under the bijection which we describe in Sec.~\ref{sbij}. Similarly, $(X_{\ell})$ are exchangeable, and $X_i$ and $X_j$ have the same distribution for all $i,j \in [1,c]$. Theorem~\ref{gam} below (stated in terms of lozenge tilings) gives limiting distributions in several asymptotic regimes. 
By symmetry, we also get similar results for the random variables $Y_{\ell}$, the number of rows containing $\ell$ in a random plane partition. Hence the results also imply some bounds (deviation) for the {\it area} containing $\ell$.



\subsection{Lozenge tilings of a hexagon}
A {\it lozenge tiling} is a tiling of a planar domain with three types of {\it lozenges} 
\begin{tikzpicture}[scale = 0.25]
\leftside{0}{0}{0}
\end{tikzpicture}
\begin{tikzpicture}[scale = 0.25]
\topside{0}{0}{0}
\end{tikzpicture} 
\begin{tikzpicture}[scale = 0.25]
\rightside{0}{0}{0}
\end{tikzpicture}
which we refer to as {\it left}, {\it top} and {\it right} tiles. 

Let $\mathrm{LT}(a, b, c)$ be the set of {lozenge tilings}  of a 
{\it hexagon} with sides $(a, b, c, a, b, c)$ as in Fig.~\ref{bir}.
For a lozenge tiling ${T} \in \mathrm{LT}(a, b, c)$, define:
\begin{itemize}
\item 
{\it left corners} (or simply corners) of ${T}$, that are local tile configurations of the form 
\begin{tikzpicture}[scale = 0.25]
\pplanepartition{{1}}
\end{tikzpicture}
\item for each corner $\alpha$, its {\it height} $h(\alpha) \in [1,c]$ is the $z$-coordinate of the top tile of $\alpha$, 
when $T$ is viewed in $\mathbb{R}^3$ as a pile of cubes (boxed plane partition), see Fig.~\ref{bir}. 
\end{itemize}


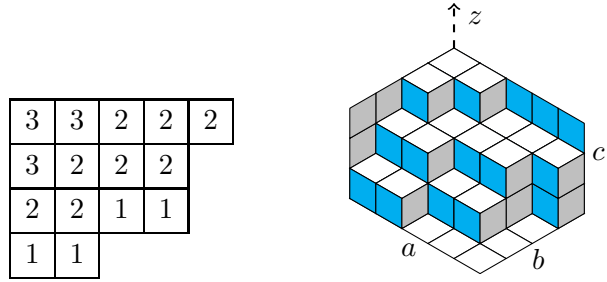
\begin{figure}
\ytableausetup{aligntableaux = bottom}
\begin{ytableau}
 3 & {3} & {2} & 2 & 2  \\
 {3} & 2 & {2} & 2\\  
 {2} & {2} & 1 & 1\\
 1 & 1
\end{ytableau}\qquad\qquad
\begin{tikzpicture}[scale = 0.4]
\foreach \c in {1,...,5} {
  \foreach \d in {1,...,3}{
	\leftside{0}{\c}{\d};
  }
}

\foreach \c in {1,...,4} {
  \foreach \d in {1,...,3}{
	\rightside{\c}{0}{\d};
  }
}

\foreach \c in {1,...,4} {
  \foreach \d in {1,...,5}{
	\topside{\c}{\d}{0};
  }
}

\planepartition{
{3,3,2,2,2},
{3,2,2,2},
{2,2,1,1},
{1,1,0,0}}

\node at (-1.5,-3.7) {$a$};
\node at (2.8,-4.0) {$b$};
\node at (4.8,-0.5) {$c$};

\draw[thick,->, dashed] (0,3) -- (0,4.5);
\node at (0.7,4) {$z$};

\end{tikzpicture}
\caption{A plane partition and the corresponding lozenge tiling of a hexagon with sides $a = 5, b = 4, c = 3$. It has 11 left corners 5 of which have height 2.}\label{bir}
\end{figure}

Consider the uniform probability measure on the set $\mathrm{LT}(a,b,c)$. We clearly have $|\mathrm{LT}(a,b,c)| = |\mathrm{PP}(a,b,c)|$, where the bijection is given by writing the height of each top tile.  
The random variables $C_{\ell}$ and $X_{\ell}$ defined for previous models translate into the following random variables for lozenge tilings.

\begin{definition}[Corner distributions]
Let $\Gamma_{\ell}$ be the number of left corners of height $\ell$.  
\end{definition}

We show the following symmetry.

\begin{theorem}[Exchangeability of levels, cf. Corollary~\ref{exch}]\label{sym1}
The random variables $(\Gamma_{\ell} )$ are exchangeable. In particular, $\Gamma_{i}$ and $\Gamma_{j}$ have the same distribution for all $i, j \in [1,c]$. 
\end{theorem}

We find limits  of corner distributions $\Gamma$ in several asymptotic regimes.

\begin{theorem}\label{gam}
We have convergence in distribution in the following regimes:

(i) Poisson: 
for $a,b,c \to \infty$ with $ab/c \to {t} > 0$ we have
\begin{align*}
\Gamma_{\ell} \to \mathrm{Poisson}(t)
\end{align*}

(ii) Negative binomial: 
for $b$ fixed and $a,c \to \infty$ with $a/(a + c) \to {q} \in (0,1)$  we have
\begin{align*}
\Gamma_{\ell} \to \mathrm{NB}(b,q)
\end{align*}

(iii) Gaussian: 
for $a,b,c \to \infty$ with $a/(b+c) \to {u} > 0$ and $b/(b + c) \to {q} \in (0,1)$ we have 
\begin{align*}
\frac{\Gamma_{\ell} - \mu N}{\sigma\sqrt{  N}}  \to  \mathcal{N}(0,1),
\end{align*}
where $N = b + c$, $\mu = {u q}$, and $\sigma = \sqrt{u (1 + u) q (1 - q)}$.

Moreover, the joint distribution of any collection of random variables $\Gamma_{\ell}$ weakly converges to the distribution of independent random variables given in corresponding regimes (i), (ii) or (iii).
\end{theorem}


Consider the Gaussian regime, say for $a = b = c = n \to \infty$. It is known that (with a probability going to 1) random lozenge tilings have the {\it arctic circle} phenomenon: 
outside of the inscribed circle of the hexagon, there are {\it frozen} regions, and inside there is a {\it liquid} region \cite{clp, ckp, joh02, ken}, see Fig.~\ref{ush}. 
Note that left corners on the topmost level at height $c$ form a profile for the upper frozen region. Then the number of such corners is about $n/2$ with fluctuations of order $\sqrt{n}$, which is compatible with known edge behavior. However what may seem surprising is the behavior across height levels. 
There is a kind of {invariance} there, 
the number of corners still has the same distribution on each level (even though some are inside the liquid region). Moreover, the theorem also tells that there is asymptotic independence between height levels which is different from the situation between slices parallel to the sides of the hexagon. (The levels $\ell$ in Theorem~\ref{gam} can first be viewed as fixed, and then using exchangeability in Theorem~\ref{sym1} we can let $\ell$ be arbitrary and grow.)

As it will be clear later, 
the arctic phenomenon also passes to the first object of $\mathbb{N}$-matrices with bounded last passage time, so that upper frozen boundary is recorded in the last column of a large matrix and the matrix itself has frozen regions in top left and bottom right quadrants, see Fig.~\ref{ush}.
\begin{figure}
\begin{tikzpicture}[scale = 0.15]
\foreach \c in {1,...,15} {
  \foreach \d in {1,...,15}{
	\leftside{0}{\c}{\d};
  }
}

\foreach \c in {1,...,15} {
  \foreach \d in {1,...,15}{
	\rightside{\c}{0}{\d};
  }
}

\foreach \c in {1,...,15} {
  \foreach \d in {1,...,15}{
	\topside{\c}{\d}{0};
  }
}

\planepartition{
{15, 15, 15, 15, 15, 15, 15, 15, 15, 14, 14, 13, 12, 10, 7}, 
{15, 15, 15, 15, 15, 15, 14, 14, 12, 12, 12, 12, 9, 7, 5}, 
{15, 15, 15, 15, 14, 14, 14, 13, 12, 12, 10, 8, 7, 5, 2}, 
{15, 15, 15, 14, 14, 14, 13, 13, 11, 11, 8, 7, 7, 4, 2}, 
{15, 15, 14, 14, 13, 13, 13, 11, 11, 11, 7, 7, 7, 4, 2}, 
{15, 15, 14, 14, 13, 13, 13, 11, 10, 9, 7, 5, 4, 4, 0}, 
{15, 14, 13, 13, 12, 12, 12, 10, 9, 9, 7, 5, 4, 2, 0}, 
{15, 14, 12, 12, 11, 9, 9, 8, 6, 5, 5, 3, 3, 2, 0},
{14, 13, 11, 10, 9, 8, 8, 8, 5, 5, 5, 1, 1, 0, 0}, 
{14, 13, 11, 9, 9, 6, 6, 5, 5, 5, 5, 1, 1, 0, 0}, 
{13, 12, 10, 9, 6, 5, 5, 5, 5, 5, 4, 1, 0, 0, 0}, 
{13, 12, 9, 6, 5, 5, 4, 4, 4, 4, 2, 0, 0, 0, 0}, 
{12, 11, 7, 5, 5, 4, 4, 4, 4, 3, 1, 0, 0, 0, 0}, 
{12, 10, 6, 4, 4, 4, 3, 3, 2, 1, 1, 0, 0, 0, 0}, 
{8, 7, 4, 4, 4, 3, 3, 3, 1, 1, 1, 0, 0, 0, 0}}

\draw[thick] (0,0) circle (12.9);
\end{tikzpicture}
\begin{tikzpicture}[scale = 0.24]
\draw[thick, dashed] (27.2,0) circle (7.5);
\fill[white] (27.2,0) rectangle (27.2+10,10);
\fill[white] (27.2,0) rectangle (27.2-10,-10);
\node at (27, 0) 
{
{\scriptsize 
\def\arraystretch{0.5}
\setlength\tabcolsep{1.3pt}
\begin{tabular}{|ccccccccccccccc|}
\hline
~ & ~ & ~ & ~ & ~ & ~ & 1 & 0 & 0 & 1 & 0 & 1 & 1 & 2 & 3\\
~ & ~ & ~ & ~ & 1 & 0 & 1 & 0 & 1 & 0 & 0 & 2 & 0 & 1 & 2\\
~ & ~ & ~ & ~ & 1 & 0 & 0 & 1 & 0 & 1 & 0 & 2 & 0 & 1 & 1\\
~ & ~ & ~ & ~ & 0 & 0 & 0 & 1 & 0 & 0 & 0 & 0 & 1 & 2 & 1\\
~ & 1 & 0 & 0 & 0 & 0 & 2 & 0 & 0 & 0 & 2 & 0 & 0 & 0 & 0\\
~ & 0 & 0 & 1 & 0 & 0 & 0 & 0 & 0 & 1 & 1 & 0 & 3 & 2 & 1\\
~ & 0 & 0 & 1 & 1 & 0 & 1 & 0 & 2 & 1 & 0 & 3 & 2 & 0 & 0\\
~ & 1 & 2 & 0 & 0 & 1 & 0 & 0 & 2 & 0 & 1 & 2 & 0 & 1 & 1\\
~ & 0 & 0 & 0 & 0 & 0 & 0 & 3 & 0 & 1 & 0 & 0 & 0 & 0 & ~\\
1 & 0 & 0 & 0 & 1 & 2 & 0 & 0 & 1 & 0 & 1 & 0 & 1 & 1 & ~\\
1 & 0 & 0 & 1 & 4 & 1 & 0 & 0 & 1 & 1 & 0 & 0 & 0 & ~ & ~\\
0 & 1 & 0 & 1 & 1 & 1 & 0 & 0 & 1 & 0 & 0 & 1 & 1 & ~ & ~\\
0 & 0 & 1 & 3 & 2 & 0 & 1 & 0 & 0 & 0 & 1 & 0 & ~ & ~ & ~\\
0 & 1 & 0 & 1 & 0 & 1 & 0 & 0 & 0 & 1 & 0 & 1 & ~ & ~ & ~\\
3 & 0 & 3 & 3 & 0 & 0 & 1 & 1 & ~ & ~ & ~ & ~ & ~ & ~ & ~\\
\hline
\end{tabular}
}
};
\end{tikzpicture}
\caption{A random lozenge tiling from $\mathrm{LT}(15,15,15)$ and the corresponding matrix from $\mathrm{BM}(15,15,15)$ whose blank (frozen) parts contain only zeros.}\label{ush}
\end{figure}
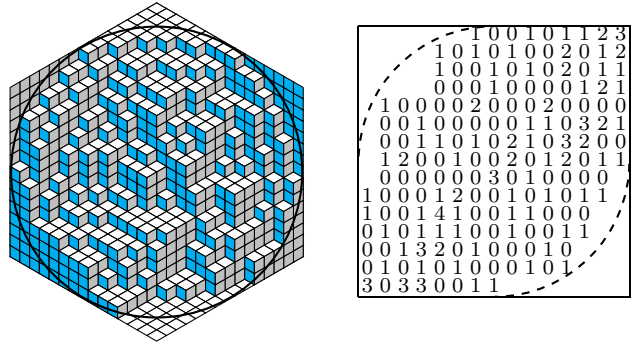

There is a general CLT for linear statistics on {\it determinantal point processes} with Hermitian symmetric correlation kernels, see \cite{sos}.  
Even though overall distributions of lozenges form a {determinantal process}, its kernel has a non-Hermitian form, see \cite{petrov}, which means that the general result cannot be applied to every lozenge statistics. 
For certain height-type statistics, asymptotic normality with logarithmic fluctuations was proved in \cite{bf}.

A well studied statistics on lozenge tilings is distributions of a particular lozenge type across corresponding slices parallel to the sides of the hexagon, see \cite{bp, petrov, gp} and references therein. In our case we consider {\it right} lozenges \begin{tikzpicture}[scale = 0.25]
\rightside{0}{0}{0}
\end{tikzpicture} (which complement left corners). We also determine joint distributions of $\Gamma$ with that statistics, see Sec.~\ref{joint}. 

We prove our results about distributions of $\Gamma$ by showing that their probability generating function can be expressed via {\it normalized} Schur polynomials which in turn follows from a result on {dual Grothendieck polynomials} defined below. 
\begin{theorem}[cf. Theorem~\ref{pgenf}]
The {\it probability generating function} of $\Gamma$ has the following explicit formula
$$
\sum_{\gamma_1, \ldots, \gamma_c \in \mathbb{N}} \mathrm{P}(\Gamma_1 = \gamma_1, \ldots, \Gamma_c = \gamma_c)\, x_1^{\gamma_1} \cdots x_c^{\gamma_c} 
= \frac{s_{(a^b)}(1^{b}, x_1, \ldots, x_c)}{s_{(a^b)}(1^{b+c})}.
$$
\end{theorem}

\subsection{Dual Grothendieck polynomials} 
Our results crucially rely on properties of these symmetric functions. The {\it dual symmetric Grothendieck polynomials} $g_{\lambda}$ can be defined via the following combinatorial presentation
$$
g_{\lambda}(x_1, x_2, \ldots ) = \sum_{T \in \mathrm{PP}(\lambda)} \prod_{i \ge 1} x_i^{c_i(T)},
$$
where the sum runs over {plane partitions} $T$ of shape $\lambda$ and $c_i(T)$ is the number of columns of $T$ containing $i$.
These polynomials were explicitly introduced and studied in \cite{lp} (and earlier implicitly in \cite{lenart, buch}) and they are related to $K$-homology of Grassmannians. 
More properties of these functions can be found in \cite{dy, dy2}.
The polynomials $g_{\lambda}$ can be considered as an inhomogeneous $K$-theoretic deformation of Schur polynomials. By definition, it is easy to see that the top degree homogeneous component of $g_{\lambda}$ is the Schur polynomial $s_{\lambda}$, i.e. $g_{\lambda} = s_{\lambda} + \text{lower degree terms}$. 

\subsection{$g$-measure on plane partitions} Let $q_1, \ldots, q_c \in (0,1)$. For a plane partition $\pi$ define the {\it descent set}
$$
\mathrm{Des}(\pi) := \{(i,j) : \pi_{i j} > \pi_{i+1, j} \}
$$
and let $\mathrm{des}(\pi) = |\mathrm{Des}(\pi)|$. 

Consider the probability distribution $\mathrm{P}_{g,b,c}$, which we call {\it $g$-measure} or {\it $g$-distribution}, on the (infinite) set $\mathrm{PP}(\infty, b, c)$ of plane partitions with at most $b$ rows and maximal entry at most $c$, defined as follows:
$$
\mathrm{P}_{g,b,c}(\pi) := \frac{1}{Z_{b,c}} \prod_{(i,j) \in \mathrm{Des}(\pi)} q_{\pi_{i j}},
$$
where the associated partition function (normalization) can be computed as (see Corollary~\ref{znorm})
$$
Z_{b,c} = \prod_{i = 1}^{c} (1 - q_i)^{-b}.
$$

\begin{proposition}[cf. Proposition~\ref{gme}]
Let $\pi \in \mathrm{PP}(\infty, b, c)$ and $\lambda = (\lambda_1 \ge \cdots \ge \lambda_b \ge 0)$. We have 
$$
\mathrm{P}_{g,b,c}(\mathrm{sh}(\pi) = \lambda) = \frac{1}{Z_{b,c}}\, g_{\lambda}(q_1, \ldots, q_c).
$$ 
\end{proposition}
By this formula it is convenient to view $\mathrm{P}_{g,b,c}$ as a distribution on integer partitions $\lambda$ with at most $b$ parts. We write $\mathrm{P}_{g,b,c}(\lambda)$ to mean this distribution. This measure is somewhat similar to the {\it Schur measure} \cite{ok, or} as it relies on the Cauchy-type identity 
for dual Grothendieck polynomials (which is a special case of a more general Cauchy identity involving two families of symmetric Grothendieck functions \cite{dy2}). 

\begin{theorem}[cf. Theorem~\ref{fpd}] 
For $\pi \in \mathrm{PP}(\infty, b, c)$, let $\lambda = (\lambda_1 \ge \cdots \ge \lambda_b \ge 0)$ be the shape of $\pi$.  
The distribution of the first part $\lambda_1$ can be 
expressed as 
\begin{align*} 
\mathrm{P}_{g,b,c}(\lambda_1 \le a) 
= \frac{1}{Z_{b,c}} s_{(a^b)}(1^{b}, q_1, \ldots, q_c). 
\end{align*}
\end{theorem}

In particular, this formula implies Toeplitz (and Fredholm) determinantal expressions for this distribution, see Sec.~\ref{gmeas}. 

\subsection{Corner growth model}\label{corner0}
Let $W = (w_{ij})_{i,j\ge 1}$ be a random matrix with i.i.d. entries $w_{ij}$ that have geometric distribution with parameter $q \in (0,1)$, i.e. 
$$\mathrm{P}(w_{ij} = k)  = (1 - q) q^k, \quad k \in \mathbb{N}.$$ 
Consider the {last passage times} for all $m,n \ge 1$ 
$$
G(m,n) = \max_{\Pi} \sum_{(i,j) \in \Pi} w_{ij},
$$ 
where the maximum is over monotone paths $\Pi$ from $(1,1)$ to $(m,n)$. 

The {\it corner growth model} can be viewed as evolution of a {random Young diagram} $Y(t)$ given at time $t$ by the region
$$
Y(t) = \{(i,j) \in \mathbb{N}^2 : G(i,j) \le t \}.
$$
The matrix $G = (G(i,j))_{i,j\ge 1}$ is called the {\it performance table}. This model was studied in \cite{joh1, bar}, see also \cite{joh2,  sep, romik} and references therein for more on the topic. In Sec.~\ref{corner} we review and use some known results.

The main relationship between this stochastic process and the $g$-distribution defined above is the following.
\begin{theorem}[cf. Theorem~\ref{main}]\label{main1}
Let $\lambda = (\lambda_1, \ldots, \lambda_b)$ be a partition with at most $b$ parts. 
We have
\begin{align*}
\mathrm{P}(G({b,c}) = \lambda_1, \ldots, G(1,c) = \lambda_b) = \mathrm{P}_{g,b,c}(\lambda), 
\end{align*}
where $q_i = q$ for all $i \in [1,c]$ in the $g$-measure.
\end{theorem}
This formula is useful for both objects that it relates, i.e. as a formula for row distributions of the performance table as well as a source for new properties of dual Grothendieck polynomials. 

\subsection{On methods and tools}
Random plane partitions were studied intensively and there is large literature on the subject, see \cite{clp, joh2, or, jn, ken, bp, petrov, gp} and many references therein. Yet, exploiting tools from combinatorial $K$-theory
turns out to be beneficial for establishing new properties.
Let us summarize the key ingredients crucial to our results. One of them is the bijection 
between plane partitions and $\mathbb{N}$-matrices, 
described in the following section. This map is different from the well-known Robinson-Schensted-Knuth (RSK) correspondence; its description is somewhat simpler but has analogous properties. 
Another key tool is in dual Grothendieck polynomials whose combinatorics is tied to plane partitions. 
As we see, these functions are useful for uncovering new properties and symmetries for some natural statistics on plane partitions, which seem to be difficult to obtain otherwise.
The determinantal formulas  
in certain cases allow us to make a transition to Schur polynomials, which is non-trivial at the combinatorial level. Combined with this interesting coincidence (see Lemma~\ref{l1}), we then use known asymptotics for normalized Schur polynomials from \cite{bp, gp}.

\section{A bijection between plane partitions and $\mathbb{N}$-matrices}\label{sbij}
Given a plane partition $\pi$, define the {\it descent level sets} 
\begin{align}\label{ldes}
D_{i \ell} := \{ j : \pi_{i j} = \ell  > \pi_{i+1, j} \},
\end{align}
i.e. $D_{i \ell}$ is the set of column indices of the entry $\ell$ in $i$th row of $\pi$ that are strictly larger than the entry below. 

Before proceeding, let us mention some important properties of these sets: 
\begin{itemize}
\item The elements in $D_{i \ell}$ form a consecutive segment.
\item For $(i,\ell) \ne (i_1, \ell_1)$ such that $i \ge i_1$, $\ell \ge \ell_1$, the sets $D_{i \ell}$ and $D_{i_1 \ell_1}$ are {\it disjoint} (more precisely, every element in $D_{i_1 \ell_1}$ is larger than any element of $D_{i \ell}$).
\end{itemize}

Let now $d_{i \ell} := |D_{i \ell}|$ and $D := (d_{i \ell})_{i, \ell \ge 1}$.
Define the map $\Phi : \{\text{plane partitions}\} \to \{\mathbb{N}\text{-matrices}\}$ by setting 
\begin{align}\label{piq}
\Phi(\pi) = D.
\end{align}

\begin{example}\label{ex1} Consider $\Phi: \pi \mapsto D$ as follows
\begin{center}
$\Phi : $
\ytableausetup{aligntableaux = center}
\begin{ytableau}
 4 & {4} & {2} \\
 {4} & 2 & {1} \\  
 {2} & {2}
\end{ytableau}
$~\longmapsto
\left(
\begin{matrix}
0 & 1 & 0 & 1\\
1 & 0 & 0 & 1\\
0 & 2 & 0 & 0\\
\end{matrix}
\right)
$
\end{center}
\end{example}

It is convenient to view the map $\Phi$ geometrically: present $\pi$ as a pile of cubes in $\mathbb{R}^3$; then
mark corners 
\begin{tikzpicture}[scale = 0.25]
\pplanepartition{{1}}
\end{tikzpicture}
on the surface of $\pi$ with $\bullet$ as in Fig.~\ref{figa}. Then $d_{i \ell}$ is the number of such marks of height $z = \ell$ and 
$x = i$. Alternatively, we may view this as a projection of these marks on the $y = 0$ plane, see Fig.~\ref{figa}.

\begin{figure}
\begin{tikzpicture}[scale = 0.45]
\planepartition{{4,4,2},{4,2,1},{2,2}}
\draw[thick, dashed,->] (0,4) -- (0,5);
\node at (0.5,5) {$z$};
\draw[thick, dashed,->] (-2.6,-1.5) -- (-3.6, -2.1); 
\node at (-4,-2.1) {$x$};
\draw[thick, dashed,->] (2.6,-1.5) -- (3.6, -2.1); 
\node at (4,-2.1) {$y$};

\foreach \c in {1,...,3} {
  \foreach \d in {1,...,4}{
	\rrightside{\c}{-3}{\d};
  }
}
\draw[thick, dashed,->] (-2.6-2.7,-1.5+1.5) -- (-3.6-2.7, -2.1+1.5); 
\node at (-4-2.7,-2.1+1.5) {$i$};

\draw[thick, dashed,->] (0-2.6,4+1.5) -- (0-2.6,5+1.5);
\node at (0.5-2.7,5+1.5) {$\ell$};

\node at (1.4,0.25) {{\scriptsize {$\bullet$}}};
\draw[dashed, gray] (1.4, 0.2) to (-3.5, 3);
\node at (-3.65,3.25) {{\scriptsize $1$}};
\node at (-3.65+0.2,3.25-0.25) {{\scriptsize $\bullet$}};

\node at (-1.2,-0.25) {{\scriptsize {$\bullet$}}};
\draw[dashed, gray] (1.4-2.65, 0.2-0.45) to (-3.5-2.65+1, 3-0.45-0.6);
\node at (-3.5-2.65+1-0.3, 3-0.45-0.6+0.25) {{\scriptsize $2$}};
\node at (-3.5-2.65+1-0.05, 3-0.45-0.6+0.05) {{\scriptsize $\bullet$}};

\node at (-2.1,0.25) {{\scriptsize {$\bullet$}}};

\node at (-1.2,2.75) {{\scriptsize {$\bullet$}}};
\draw[dashed, gray] (1.4-2.6, 0.2+2.5) to (-3.5-2.6+1.75, 3+2.5-1);
\node at (-3.5-2.6+1.75, 3+2.5-1) {{\scriptsize {$\bullet$}}};
\node at (-3.5-2.6+1.75-0.25, 3+2.5-1+0.25) {{\scriptsize {$1$}}};

\node at (0.5,2.75) {{\scriptsize {$\bullet$}}};
\draw[dashed, gray] (1.4-0.9, 0.2+2.5) to (-3.5-0.9+0.95, 3+2.5-0.5);
\node at (-3.5-0.9+0.95, 3+2.5-0.5) {{\scriptsize {$\bullet$}}};
\node at (-3.5-0.9+0.95-0.25, 3+2.5-0.5+0.25) {{\scriptsize {$1$}}};

\node at (-3.5-0.9+0.95-0.25-3, 3+2.5-0.5+0.25-1) {{\scriptsize {$D = (d_{i \ell})$}}};

\node at (0.5,-1.25) {{\scriptsize {$\bullet$}}};
\draw[dashed, gray] (1.4-0.9, 0.2-1.5) to (1.4-2.65, 0.2-0.45);
\node at (-3.5-2.65+1-0.25+1-0.1, 3-0.45-0.6+0.25-1+0.55) {{\scriptsize $1$}};
\node at (-3.5-2.65+1-0.25+1+0.08, 3-0.45-0.6+0.25-1+0.3) {{\scriptsize $\bullet$}};
\end{tikzpicture}
\caption{A geometric view of the map $\Phi$. The dots $\bullet$ correspond to left corners.}\label{figa}
\end{figure}
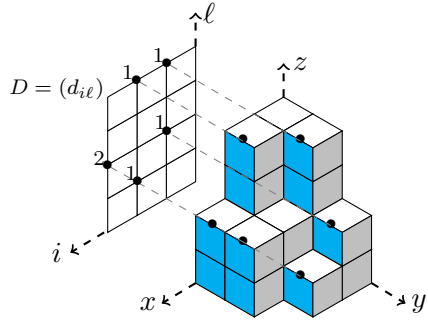

\begin{theorem}\label{bij}
The map $\Phi$ defines a bijection between the set $\mathrm{PP}(a,b,c)$ of boxed plane partitions and $\mathrm{BM}(a,b,c)$ of $\mathbb{N}$-matrices with bounded last passage time.
\end{theorem}
\begin{proof}
Suppose $\pi \in \mathrm{PP}(a,b,c)$. Let us show that $D = \Phi(\pi) \in \mathrm{BM}(a,b,c)$. By construction, it is clear that $D$ has at most $b$ rows and $c$ columns. 
Take any monotone path $(1,1) = (i_1, \ell_1) \to \cdots \to (i_k, \ell_k) \le (b,c)$ in $D$. Then the descent level sets $D_{i_s \ell_s}$ for $s = 1, \ldots,k$ are pairwise disjoint by the property mentioned above. Hence
$$
\sum_{s = 1}^k d_{i_s \ell_s} = \sum_{s = 1}^k |\{j : \pi_{i_{s} j} = \ell_s > \pi_{i_s + 1, j} \}| \le a
$$
as needed.

Let us now describe the inverse map $\Phi^{-1}$. Given an $\mathbb{N}$-matrix $D \in \mathrm{BM}(a,b,c)$, we show how to uniquely reconstruct a plane partition $\pi \in \mathrm{PP}(a,b,c)$ such that $\Phi(\pi) = D$. We will build $\pi$ sequentially by scanning the columns of $D$ starting from the last one. 

Let us show how to {\it add} an element $\ell$ in $i$-th row of some plane partition $\pi'$. To do so, find the first available column of $\pi'$ whose length is less than $i$, then add the elements $\ell$ in that column so that its length becomes $i$. For example, suppose $\pi'$ already had the following shape and we want to add $\ell$ in the 3rd row:

{\scriptsize
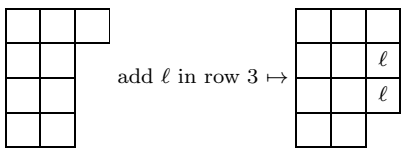
\begin{center}
\begin{ytableau}
 ~ & {~} & {~} \\
 {~} & ~  \\  
 {~} & {~} \\
 ~ & ~
\end{ytableau}
add $\ell$ in row $3$ $\mapsto$
\begin{ytableau}
 ~ & {~} & {~} \\
 {~} & ~ & {\ell} \\  
 {~} & {~} & {\ell} \\
 ~ & ~
\end{ytableau}
\end{center}
}

Let us initially set $\pi = \varnothing$. To {\it add} the $\ell$-th column $(d_{1 \ell }, \ldots, d_{b \ell })^T$ of $D$ for $\ell = c,c-1, \ldots, 1$ we do the following. For each $i = b, b-1, \ldots, 1$ add $\ell$ in $i$-th row of $\pi$ exactly $d_{i \ell}$ times. 

Let us check that after this procedure we have $\pi \in \mathrm{PP}(a,b,c)$. By construction, it is clear that $\pi$ has at most $b$ rows and the largest entry at most $c$. We know that in $D$ the maximal weight of a monotone path from $(1,1)$ to $(b,c)$ is at most $a$. 
Suppose to the contrary that the length of the first row of $\pi$ is larger than $a$. Consider the elements of the first row of $\pi$, say $(\ell_1 \ge \cdots \ge \ell_{k})$ for $k > a$ and suppose they were added in rows $i_1, \ldots,  i_k$ during the procedure. 
Then we must have $i_1 \ge \cdots \ge i_k$ and hence there is a monotone path in $D$ that passes through all points $(i_k, \ell_k) \to \cdots \to (i_{1}, \ell_{1})$ (here if we have the same point $(i_{j}, \ell_j)$ repeated several times we may use it just once as its multiplicity is recorded in $d_{i_{j} \ell_{j}}$). The weight of any such monotone path is at least $\sum_{j} d_{i_{j} \ell_{j}} \ge k > a$ which is a contradiction. Finally, it is easy to see that 
$\pi$ has the desired property $\Phi(\pi) = D$.
\end{proof}

\begin{example}\label{ex3}
Let us show how $\Phi^{-1}$ applies to the matrix in Example~\ref{ex1}. We have: 

\begin{center}
$\Phi^{-1} : D=
\left(
\begin{matrix}
0 & 1 & 0 & 1\\
1 & 0 & 0 & 1\\
0 & 2 & 0 & 0\\
\end{matrix}
\right) \longmapsto 
$
\begin{ytableau}
 4 & {4} & {2} \\
 {4} & 2 & {1} \\  
 {2} & {2} 
\end{ytableau}
\end{center}
which is obtained as follows: 

add column 4: add $4$ in row $2$:  
{\scriptsize \begin{ytableau}
 4  \\
 {4}   
\end{ytableau}
}
\quad add $4$ in row $1$:  
{\scriptsize
\begin{ytableau}
 4  & 4 \\
 {4}   
\end{ytableau}
}

add column 3: no addition

add column 2: add $2$ in row $3$ two times:\quad  
{\scriptsize 
\begin{ytableau}
 4  & 4\\
 {4}  & 2\\ 
 2 & 2
\end{ytableau}
}
\quad add $2$ in row $1$: 
{\scriptsize
\begin{ytableau}
 4  & 4 & 2\\
 {4}  & 2\\ 
 2 & 2
\end{ytableau}
}

add column 1: add $1$ in row $2$:
{\scriptsize
\begin{ytableau}
 4  & 4 & 2\\
 {4}  & 2 & 1\\ 
 2 & 2
\end{ytableau}
}
\end{example}



\begin{corollary}
$|\mathrm{PP}(a,b,c)| = |\mathrm{BM}(a,b,c)|$.
\end{corollary}

Note that by letting $a,b,c \to \infty$ one can view $\Phi$ as a bijection between plane partitions and $\mathbb{N}$-matrices.

We discuss few more properties of the bijection $\Phi$ that will be needed later. 

\begin{lemma}
\label{phip}
Let $\pi$ be a plane partition and $\Phi (\pi)= D  =(d_{i \ell})$.

(i) Let $c_{\ell}(\pi)$ be the number of columns of $\pi$ containing entry $\ell$,
and $c_{\ell}(D) = \sum_{i} d_{i \ell}$ be column sums of $D$. Then we have $c_{\ell}(\pi) = c_{\ell}(D)$ for all $\ell \ge 1$.
  
(ii) Let $\lambda = \mathrm{sh}(\pi)$ be the shape of $\pi$. We have for all $k \ge 1$
\begin{align}\label{lamm}
\lambda_k  = \max_{\Pi : (k, 1) \to (b, c)} \sum_{(i,\ell) \in \Pi} d_{i \ell},
\end{align}
where the maximum is taken over monotone paths 
$\Pi$ from $(k, 1)$ to $(b,c)$, if $D$ has $b$ rows and $c$ columns. 
\end{lemma}
\begin{proof}
The part (i) is clear by construction. Indeed, by definition of $\Phi$ we have
$$
c_{\ell}(D) = \sum_{i} d_{i \ell} = \sum_{i} |\{ j : \pi_{i j} = \ell > \pi_{i + 1, j} \}| = c_{\ell}(\pi).
$$

(ii) Take any monotone path $\Pi$ from $(k,1)$ to $(b,c)$. Then the descent level sets $D_{i \ell}$ for $(i,\ell) \in \Pi$ are pairwisely disjoint. Using this and  since $i \ge k$ for all $(i,\ell) \in \Pi$, we obtain
\begin{equation}\label{lal}
\sum_{(i,\ell) \in \Pi} d_{i \ell} = \sum_{(i,\ell)} |\{ j : \pi_{i j} = \ell > \pi_{i + 1, j} \}| \le \lambda_k.
\end{equation}
On the other hand, suppose the $k$-th row of $\pi$ has entries $(\ell_1 \ge \cdots \ge \ell_{m})$ where $m = \lambda_k$. Assume the entries $\ell_1, \ldots, \ell_m$ end in rows $i_1 \ge \cdots \ge i_m$ of $\pi$. Then there is a monotone path $\Pi$ from $(k,1)$ to $(b,c)$ that contains all points $(i_m, \ell_m), \ldots, (i_1, \ell_1)$. The weight of any such path is at least 
$\sum_{j} d_{i_j \ell_j} \ge m = \lambda_k$. Combining this with the inequality \eqref{lal} we obtain \eqref{lamm}. 
\end{proof}

\begin{remark}
The RSK 
correspondence also gives a bijection between plane partitions and $\mathbb{N}$-matrices, see e.g. \cite[Ch.~7]{sta}. The bijection $\Phi$ has another (simpler) description and different yet related properties. We explore more properties and applications of this bijection in \cite{dy4b}. 
\end{remark}

\begin{remark}
Many RSK-type dynamics were described in \cite{bp2} and it would be interesting to see how the map $\Phi$ can be related to that classification.  
\end{remark}

\section{Dual Grothendieck polynomials} 
Recall that a {\it partition} is a  sequence $\lambda = (\lambda_1, \ldots, \lambda_{\ell})$ of nonincreasing positive integers. Here $\ell = \ell(\lambda)$ is the {\it length} of $\lambda$. The {\it Young diagram} of $\lambda$ is the set $\{(i,j)  : i \in [1, \ell], j\in[1, \lambda_i] \}$. We denote by $\lambda'$  the {\it conjugate} partition of $\lambda$, which has transposed Young diagram.

\begin{definition}[\cite{lp}]
Let $\mathrm{PP}_n(\lambda)$ be the set of plane partitions of shape $\lambda$ and largest entry at most $n$. The {\it dual symmetric Grothendieck polynomials} $\{ g_{\lambda} \}$ are defined by the following combinatorial formula
$$
g_{\lambda}(x_1, \ldots, x_n) = \sum_{T \in \mathrm{PP}_n(\lambda)} x^T,
$$
where $x^T = \prod_{i = 1}^n x_i^{c_i}$ and $c_i$ is the number of columns of $T$ containing $i$.
\end{definition}
The following properties hold: $g_{\lambda}$ is an inhomogeneous symmetric polynomial whose top degree component is the Schur polynomial, i.e. $g_{\lambda} = s_{\lambda} + \text{lower degree elements}$. 

Besides combinatorial definition we need the following determinantal formula. We denote $\mathbf{x} = (x_1, x_2, \ldots )$ and $\{e_n\}$ are the elementary symmetric functions. 
\begin{lemma}[Jacobi-Trudi identity, see \cite{dy}]
We have
\begin{equation}\label{jt}
g_{\lambda}(\mathbf{x}) = \det\left[ e_{\lambda'_i - i + j}(1^{\lambda'_i - 1}, \mathbf{x})  \right]_{1 \le i,j \le \lambda_1}.
\end{equation}
\end{lemma}

\begin{remark}
Jacobi-Trudi formulas for $g_{\lambda}$ were first obtained in \cite{sz}. A proof was given in \cite{dy}. More properties of these polynomials can be found in \cite{lp, dy, dy2}.
\end{remark}

The following properties will be important for  our results.
\begin{lemma}[Coincidence lemma]\label{l1}
Let $\rho := (a^{b})$. We have 
$$
g_{\rho}(\mathbf{x}) = s_{\rho}(1^{b - 1}, \mathbf{x}).
$$
\end{lemma}
\begin{proof}
We show that the Jacobi-Trudi formulas for both polynomials coincide under these conditions. Note firstly that by the classical Jacobi-Trudi formula we have
\begin{align}\label{sr}
s_{\rho}(1^{b - 1}, \mathbf{x}) = \det\left[ e_{b - i + j}(1^{b - 1}, \mathbf{x})\right]_{1 \le i, j \le a}
\end{align}
For the polynomials $g$, the Jacobi-Trudi identity \eqref{jt} 
for $\lambda = \rho$ gives \eqref{sr}.
\end{proof}

\begin{remark}
While this coincidence formula is evident from determinantal identities, we should point out that it is {\it not} at all obvious combinatorially. Here one would use an argument as in \cite{dy} on combinatorial proof of the identity \eqref{jt}. 
\end{remark}

\begin{corollary}
We have the following formula
\begin{align}
\sum_{\lambda \subset (a^{b})} g_{\lambda}(\mathbf{x}) &= s_{(a^b)}(1^{b}, \mathbf{x})\label{gx} 
\end{align}
\end{corollary}
\begin{proof}
Let $\rho = (a^{b})$. By combinatorial definition of $g$, we have the branching formula
$$
g_{\rho}(1, \mathbf{x}) = \sum_{\lambda \subset \rho} g_{\lambda}(\mathbf{x}).
$$
On the other hand, $g_{\rho}(\mathbf{x}, 1) = s_{\rho}(1^{b}, \mathbf{x})$ by Lemma~\ref{l1}. 
\end{proof}

\begin{lemma}
The following identity holds
\begin{align}\label{gl}
\sum_{\ell(\lambda) \le b} g_{\lambda}(\mathbf{x}) = \prod_{n}^{} {(1 - x_n)^{-b}}
\end{align}
\end{lemma}

\begin{proof}
This identity follows from the Cauchy identity for the pair $G_{\lambda}, g_{\lambda}$ of Grothendieck polynomials and the following property of the symmetric Grothendieck polynomials $G_{\lambda}(1^{b}) = 1$ if $\ell(\lambda) \le b$ and $0$ otherwise, see \cite{dy3}. In fact, one can also derive it from the bijection $\Phi$, we address this in \cite{dy4b}.
\end{proof}

\begin{remark}
The formulas \eqref{gx}, \eqref{gl} can perhaps be compared to the following identities for Schur functions, see \cite{mac}
\begin{align*}
\sum_{\lambda \subset (a^{b})} s_{\lambda}(x_1, \ldots, x_{b}) &= \frac{\det\left[ x_j^{a + 2b - i} - x_j^{i - 1} \right]_{1 \le i, j \le b}}{\det\left[ x_j^{2b - i} - x_j^{i - 1} \right]_{1 \le i, j \le b}},\\
\sum_{\ell(\lambda) \le b} s_{\lambda}(x_1, \ldots, x_b) &= \prod_{1 \le n \le b}\frac{1}{1 - x_n} \prod_{1 \le i < j \le b} \frac{1}{1 - x_i x_j}.
\end{align*}
\end{remark}

\section{Normalized Schur polynomials}\label{normschur}
The {\it normalized Schur polynomial} $S_{\lambda}(\mathbf{x}; N)$ 
is defined for $N \ge \max(k, \ell(\lambda))$ as follows 
\begin{align}
S_{\lambda}(x_1, \ldots, x_k; N) := \frac{s_{\lambda}(x_1, \ldots, x_k, 1^{N-k})}{s_{\lambda}(1^N)}
\end{align}
We are interested in a special case when $\lambda= (a^b)$ has rectangular shape. 
We present some known asymptotics for normalized Schur polynomials.
\begin{lemma}[see \cite{bp}]\label{bpt}
Let $N = b + c$,  $\rho = (a^b)$ and $x_i \in \mathbb{C}$. We have:

For $a,b,c \to \infty$ with $ab/c \to t > 0$
\begin{align*} 
S_{\rho}(x_1, \ldots, x_k; N) \to e^{-kt} \prod_{i = 1}^k e^{t x_i}
\end{align*}

For $b$ fixed and $a, c \to \infty$ with $a/(a + c) \to q \in (0,1)$ 
\begin{align*} 
S_{\rho}(x_1, \ldots, x_k; N) \to (1 - q)^{b k} \prod_{i = 1}^k {(1 - q x_i)^{-b}}
\end{align*}
\end{lemma}

Consider now the following {\it Gaussian} asymptotic regime. Let $u > 0, q \in (0,1)$ be real parameters. Suppose $a = \lfloor uN\rfloor,$ $b = \lfloor qN \rfloor $, $b + c = N$ and let $N \to \infty$. 
Then the partitions $\rho = (a^b)$ (when each row is rescaled by $N$) have the following simple {limit shape} $f : [0,1] \to \mathbb{R}$ given by 
$$
f(x) = 
	\begin{cases}
		u, & \text{ if } x \le q,\\
		0, & \text{ if } x > q.
	\end{cases}
$$ 
Define two more parameters
\begin{align}\label{mv}
m = \int_{0}^{1} f(x) dx = uq, \quad  v = \int_{0}^{1} f(x)^2 dx - m^2 
+ \int_{0}^{1} f(x)(1 - 2x) dx = u(u+1)q(1 - q).
\end{align}
\begin{lemma}[\cite{gp}]\label{gp}
Let $a,b,c$ be in the Gaussian regime as above. 
As $N \to \infty$, we have  
\begin{align*}
S_{\rho}(e^{x_1/\sqrt{v N}}, \ldots, e^{x_k/\sqrt{v N}}; N) \cdot \prod_{i = 1}^{k} e^{- x_i {m\sqrt{N/v}}} \to \prod_{i = 1}^{k} e^{x_i^2/2}
\end{align*}
uniformly on compact subsets of $(\mathbb{R} \setminus \{0\})^k$.
\end{lemma}
\begin{remark}
In fact, this is a special case of more general results proved in \cite{gp}.
\end{remark}

\section{Multivariate generating function for corner distributions}
\subsection{Distributions of $\Gamma_{\ell}$ in lozenge tilings} Recall that we consider the uniform probability measure on the set of lozenge tilings $\mathrm{LT}(a,b,c)$, and $\Gamma_{\ell}$ is the number of left corners of height $\ell \in [1,c]$ in a random lozenge tiling. 
Throughout the section (and later) we set $\rho = (a^b)$, $N = b + c$, and $\Gamma^{(c)} = (\Gamma_1, \ldots, \Gamma_c)$.
\begin{theorem}\label{pgenf}
We have the generating function
\begin{align}\label{mul}
\sum_{\gamma \in \mathbb{N}^c} 
\mathrm{P}(\Gamma^{(c)} = \gamma)\,  
\mathbf{x}^{\gamma} = S_{\rho}(x_1, \ldots, x_c; N).
\end{align}
\end{theorem}
\begin{proof}
By definition of $g_{\lambda}$ and the bijection between lozenge tilings and plane partitions so that $\Gamma_i \mapsto X_i$ we have 
\begin{align}
\sum_{\gamma \in \mathbb{N}^c} 
\mathrm{P}(\Gamma^{(c)} = \gamma)\,  
\mathbf{x}^{\gamma} 
= \sum_{\gamma \in \mathbb{N}^c} \mathrm{P}(X_1 = \gamma_1, \ldots, X_c = \gamma_c)\, x^{\gamma} = \frac{1}{Z_{abc}} \sum_{\lambda \subset \rho} g_{\lambda}(x_1, \ldots, x_c).
\end{align}
Using \eqref{gx} and the fact that $Z_{abc} = s_{\rho}(1^{b+c})$ we obtain
$$
\frac{1}{Z_{abc}} \sum_{\lambda \subset \rho} g_{\lambda}(x_1, \ldots, x_c) = \frac{s_{\rho}(1^b, x_1, \ldots, x_c)}{s_{\rho}(1^{b+c})} = S_{\rho}(x_1, \ldots, x_c; N).
$$
as needed.
\end{proof}

\begin{corollary}[Symmetry, exchangeability]\label{exch}
$\mathrm{P}(\Gamma^{(c)} = \gamma)$ is invariant under permutations of  the coordinates $(\gamma_1, \ldots, \gamma_c)$. 
\end{corollary}

\begin{corollary}\label{cor1} Let $k \in [1,c]$. Setting $x_{k+1}= \cdots = x_c = 1$ we have
$$
\sum_{\gamma \in \mathbb{N}^k} 
\mathrm{P}(\Gamma^{(k)} = \gamma)\,  
\mathbf{x}^{\gamma} = S_{\rho}(x_1, \ldots, x_k; N).
$$
More generally, for $\{\ell_1, \ldots, \ell_k \} \subset [1,c]$ we also have
\begin{align}\label{mulp}
\sum_{\gamma \in \mathbb{N}^k} 
\mathrm{P}(\Gamma_{\ell_1} = \gamma_1, \ldots, \Gamma_{\ell_k} = \gamma_k )\,  
\mathbf{x}^{\gamma} = S_{\rho}(x_{\ell_1}, \ldots, x_{\ell_k}; N).
\end{align}
\end{corollary}

\begin{corollary}\label{corsin} For a single variable we have
\begin{align}\label{single}
\sum_{n} \mathrm{P}(\Gamma_{\ell} = n)\, x^n = S_{\rho}(x; N). 
\end{align}
In particular, $\mathrm{P}(\Gamma_i = n) = \mathrm{P}(\Gamma_j = n)$ for all $i,j \in [1,c]$.
\end{corollary}

We are now going to prove the limiting distributions for $\Gamma$ stated in Theorem~\ref{gam}. 
\begin{proof}[Proof of Theorem~\ref{gam}]
Suppose we have collection of random variables $\Gamma := (\Gamma_{\ell_1}, \ldots, \Gamma_{\ell_k} )$. 

In the first Poisson regime $a,b \to \infty$ with $ab/c \to t$ by Lemma~\ref{bpt} we have
$$
S_{\rho}(x_{\ell_1}, \ldots, x_{\ell_k}; N) \to e^{-kt} \prod_{i = 1}^{k} e^{x_{\ell_i}}. 
$$
Hence by \eqref{mulp}, $\Gamma$ weakly converges to $k$ 
i.i.d. Poisson random variables with parameter $t$.

In the second negative binomial regime for fixed $b$ and $a,c \to \infty$ with $a/(a  + c) \to q \in (0,1)$ again by Lemma~\ref{bpt} we have 
$$
S_{\rho}(x_{\ell_1}, \ldots, x_{\ell_k}; N) \to (1 - q)^{bk} \prod_{i = 1}^{k} (1 - x_{\ell_i} q)^{-b} 
$$
Hence by \eqref{mulp}, $\Gamma$ weakly converges to $k$
i.i.d. negative binomial random variables with parameters $b,q$.

Finally, in the third Gaussian regime by \eqref{single} and Lemma~\ref{gp} for $k = 1$ 
we obtain that  $\frac{\Gamma_{\ell} - m N}{ \sqrt{v N}}$ weakly converges to $N(0,1)$ as $N \to \infty$, where $v = \sigma^2$. By Lemma~\ref{gp} and \eqref{mulp}, we get factorization of the characteristic functions in the limit which is equivalent to asymptotic independence. Hence the collection of variables  $\frac{\Gamma_{\ell_i} - m N}{\sqrt{v N}}$ become asymptotically independent as $N \to \infty$.
\end{proof}

\subsection{Distributions of $X_{\ell}$ in plane partitions and marginals $C_{\ell}$ in $\mathbb{N}$-matrices}
The results on $\Gamma_{\ell}$ easily translate to results on $X_{\ell}$ on random boxed plane partitions and $C_{\ell}$ on $\mathbb{N}$-matrices with bounded last passage time, discussed in the introduction. 

Take a uniformly random plane lozenge tilings $T \in \mathrm{LT}(a,b,c)$ and let $\Psi$ be the map that records heights of all top tiles, the result is clearly a plane partition $\pi = \Psi(T) \in \mathrm{PP}(a,b,c)$ and $\Psi$ is a bijection. We have $\Gamma_{\ell}(D) = X_{\ell}(\pi)$. 

Now take a uniformly random matrix $D \in \mathrm{BM}(a,b,c)$, then by Lemma~\ref{phip} (i), we have $C_{\ell}(D) = X_{\ell}(\pi)$ for $\pi = \Phi^{-1}(D)$. 

\section{Joint distributions of left corners and right lozenges}\label{joint}
\begin{definition}[Right lozenge distributions]
Let $Y^{(k)}$ be the vector of positions i.e. $y$-coordinates 
of the right tiles  
\begin{tikzpicture}[scale = 0.25]
\rightside{0}{0}{0}
\end{tikzpicture} 
 along $k$-th  
 slice parallel to the side $a$ of the hexagon, see Fig.~\ref{eki}. For $k \le \min(b,c)$, these positions have the form 
 $$
 \mu_1 \ge \cdots \ge \mu_k \ge 0.
 $$
\end{definition}

\begin{figure}
\begin{tikzpicture}[scale = 0.45]
\foreach \c in {1,...,5} {
  \foreach \d in {1,...,3}{
	\leftside{0}{\c}{\d};
  }
}

\foreach \c in {1,...,4} {
  \foreach \d in {1,...,3}{
	\rightside{\c}{0}{\d};
  }
}

\foreach \c in {1,...,4} {
  \foreach \d in {1,...,5}{
	\topside{\c}{\d}{0};
  }
}

\planepartition{
{3,3,2,2,2},
{3,2,2,2},
{2,2,1,1},
{1,1,0,0}}

\node at (-1.5,-3.7) {$a$};
\node at (2.8,-4.0) {$b$};
\node at (4.8,-0.5) {$c$};

\draw[thick,->, dashed] (4.3,-2.5) -- (6,-3.5);
\node at (6.5,-3) {$y$};

\draw[thick, dashed, darkgray] (-5,1.8) -- (5,-3.9);
\node at (5.4,-4.3) {$k = 3$};

\node at (2.2,-2.3) {$\bullet$};
\node at (-0.4,-0.75) {$\bullet$};
\node at (-3,0.7) {$\bullet$};

\end{tikzpicture}
\caption{An example of $k$-th slice parallel to $a$ and the corresponding right tiles marked with $\bullet$. Here $Y^{(3)} = (4,2,0)$, the $y$-coordinates of marked tiles.}\label{eki}
\end{figure}
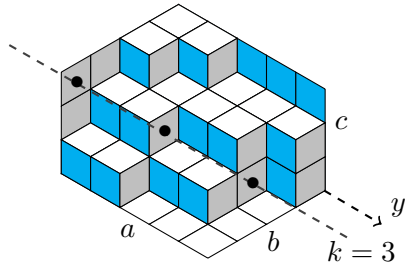

Let $\mathcal{P}_{n,k} := \{\mu : \mu\vdash n,\, \ell(\mu) \le k \}$ be the set of partitions of $n$ with length at most $k$.
We have the following Schur expansion (see \cite{bp}) of normalized Schur polynomials
\begin{align}\label{sno}
S_{\rho}(x_1, \ldots, x_k; N) = \sum_{\mu \in \mathcal{P}_{n,k}} \mathrm{P}(Y^{(k)} = \mu) \frac{s_{\mu}(x_1, \ldots, x_k)}{s_{\mu}(1^k)} 
\end{align}
where $\rho = (a^b)$, $N  = b + c$, and $k \le \min(b,c)$.
The distribution of $Y$ has the formula (e.g.~\cite{bp}) 
$$
\mathrm{P}(Y^{(k)} = \mu) = \mathrm{const}_{a,b,c,k} \prod_{1 \le i < j \le k} (\mu_i - i - \mu_j + j)^2 \prod_{i = 1}^k w(\mu_i + k - i), 
$$
where 
$$
w(y) = \frac{(a + c - 1 - y)!\, (b - k + y)!}{y!\, (a + k - 1 - y)!}
$$
which is a case of more general {\it orthogonal polynomial ensembles}.
Alternatively, there is also another formula (e.g. \cite[Prop.~5.3]{gp})
\begin{align}\label{yf}
\mathrm{P}(Y^{(k)} = \mu) = \frac{1}{Z_{abc}} s_{\mu}(1^k) s_{\mu^c}(1^{b + c - k}),
\end{align}
where $\mu^c = (a - \mu_b, \ldots, a - \mu_{1})$ is the complement to $\mu$ in  $\rho$.

It is proved in \cite{jn, gp} that for every $k$, as $N \to \infty$ we have 
convergence in distribution 
$$
\frac{1}{\sqrt{v N}}{(Y^{(k)} - (m N, \ldots, mN))} 
\to \mathrm{GUE}_k,
$$
where $\mathrm{GUE}_k$ is the distribution of the spectrum 
of a random $k \times k$ Hermitian matrix from the {\it Gaussian unitary ensemble} (GUE), and $m, v$ are defined in \eqref{mv} (cf. Lemma~\ref{gp}). 

\

Now we would like to relate the two distributions $\Gamma$ and $Y$. 

Recall that the {\it Kostka numbers} $K_{\lambda \gamma} \in \mathbb{N}$ (see e.g. \cite{mac, sta}) are coefficients defined from the monomial expansion of Schur polynomials
$$
s_{\lambda} = \sum_{\gamma} K_{\lambda \gamma}\, \mathbf{x}^{\gamma}.
$$

Consider the random vectors
$$\Gamma^{(k)} 
= (\Gamma_1, \ldots, \Gamma_k) 
$$ 
\begin{theorem}[Joint distribution of left corners and right tiles]
Let $\gamma = (\gamma_1, \ldots, \gamma_k) \in \mathbb{N}^k$ with $|\gamma| := \gamma_1 + \cdots + \gamma_k = n$ for $k \le \min(b,c)$ and $\mu \in \mathcal{P}_{n,k}$. We have: 

(i) The joint distribution of $\Gamma^{(k)}$ and $Y^{(k)}$ is given by 
\begin{align*}
\mathrm{P}(\Gamma^{(k)} = \gamma,\, Y^{(k)} = \mu ) 
= \frac{1}{Z_{a b c}} K_{\mu \gamma}\, s_{\mu^{c}}(1^{b + c - k}),
\end{align*}
where $K_{\mu \gamma}$ is the Kostka number and $\mu^c = (a - \mu_b, \ldots, a - \mu_{1})$ is complement to $\mu$ in  $(a^b)$.

\vspace{0.5em}

(ii) The distribution of $\Gamma^{(k)}$ can be computed as follows
\begin{align*}
\mathrm{P}(\Gamma^{(k)} = \gamma ) = \sum_{\mu\, \in\, \mathcal{P}_{n,k}} p_{\mu}(\gamma)\, \mathrm{P}(Y^{(k)} = \mu ),  
\end{align*}
where the conditional probability is determined as
\begin{align*}
p_{\mu}(\gamma) = \mathrm{P}(\Gamma^{(k)} = \gamma\, |\, Y^{(k)} = \mu)= \frac{K_{\mu\, \gamma}}{s_{\mu}(1^k)}.
\end{align*}
\end{theorem}

\begin{proof}
By Corollary~\ref{cor1} 
and the Schur polynomial expansion \eqref{sno} we have 
\begin{align*}
\sum_{\gamma \in \mathbb{N}^k} \mathrm{P}(\Gamma^{(k)} = \gamma)\, x^{\gamma} &= S_{\rho}(x_1, \ldots, x_k; N)\\ 
&= \sum_{\mu \in \mathcal{P}_{n,k}} \mathrm{P}(Y^{(k)} = \mu) \frac{s_{\mu}(x_1, \ldots, x_k)}{s_{\mu}(1^k)}\\ 
&=  \sum_{\mu \in \mathcal{P}_{n,k}} \mathrm{P}(Y^{(k)} = \mu) \sum_{\alpha} \frac{K_{\mu \alpha}}{s_{\mu}(1^k)} x^{\alpha}.
\end{align*}
Therefore, 
$$
\mathrm{P}(\Gamma^{(k)} = \gamma) = \sum_{\mu \in \mathcal{P}_{n,k}} \frac{K_{\mu \gamma}}{s_{\mu}(1^k)} \mathrm{P}(Y^{(k)} = \mu),
\qquad
\frac{K_{\mu \gamma}}{s_{\mu}(1^k)} = \mathrm{P}(\Gamma^{(k)} = \gamma\, |\, Y^{(k)} = \mu).
$$
To find the joint distribution we have 
\begin{align*}
\mathrm{P}(\Gamma^{(k)} = \gamma, Y^{(k)} = \mu) &= \mathrm{P}(\Gamma^{(k)} = \gamma\, |\, Y^{(k)} = \mu) \cdot \mathrm{P}(Y^{(k)} = \mu) \\
&= \frac{K_{\mu \gamma}}{s_{\mu}(1^k)} \cdot \frac{s_{\mu}(1^k) s_{\mu^c}(1^{b + c - k})}{Z_{abc}} \\
&= \frac{K_{\mu \gamma}\, s_{\mu^c}(1^{b + c - k})}{Z_{abc}}
\end{align*}
where we used the formula \eqref{yf}.
\end{proof}

Now using some properties of Kostka numbers we obtain the following inequalities. 
\begin{corollary}
Since $K_{\mu \mu} = 1$  we get the following inequality for $\mu \in \mathcal{P}_{n,k}$
$$
s_{\mu}(1^k)\, \mathrm{P}(\Gamma^{(k)} = \mu ) \ge \mathrm{P}(Y^{(k)} = \mu).
$$
\end{corollary}

For $\alpha, \beta \in \mathbb{N}^{k}$, say that $\beta$ {\it dominates} $\alpha$ and write $\beta \succeq \alpha$ if $\sum_{i = 1}^{\ell} \beta_i \ge \sum_{i = 1}^{\ell} \alpha_i$ for all $\ell \in [1,k]$ and $\sum_{i = 1}^k \alpha_i = \sum_{i = 1}^k \beta_i$.

\begin{corollary}[Monotonicity]
The following inequality holds
$$
\forall\, \beta \succeq \alpha: \quad \mathrm{P}(\Gamma^{(k)} = \alpha ) \ge \mathrm{P}(\Gamma^{(k)} = \beta ).
$$
\end{corollary}
\begin{proof} 
It is known that $K_{\mu \alpha} \ge K_{\mu \beta}$ for $\beta \succeq \alpha$ (see \cite{mac}). Hence $p_{\mu}(\alpha) \ge p_{\mu}(\beta)$ and subsequently $\mathrm{P}(\Gamma^{(k)} = \alpha) \ge \mathrm{P}(\Gamma^{(k)} = \beta)$. 
\end{proof}

\section{$g$-measure}\label{gmeas}
Let $q_1, \ldots, q_c \in (0,1)$. Recall that we define the probability distribution, called {\it $g$-measure} or $g$-distribution, on the (infinite) set $\mathrm{PP}(\infty, b, c)$ of plane partitions with at most $b$ rows and maximal entry at most $c$
$$
\mathrm{P}_{g,b,c}(\pi) = \frac{1}{Z_{b,c}} \prod_{(i,j) \in \mathrm{Des}(\pi)} q_{\pi_{i j}},
$$
where $Z_{b,c}$ is the normalization constant. 

\begin{proposition}\label{gme}
Let $\pi \in \mathrm{PP}(\infty, b, c)$ and $\lambda = (\lambda_1 \ge \cdots \ge \lambda_b \ge 0)$. We have  
$$
\mathrm{P}_{g,b,c}(\mathrm{sh}(\pi) = \lambda) = \frac{1}{Z_{b,c}}\, g_{\lambda}(q_1, \ldots, q_c).
$$ 
\end{proposition}
\begin{proof}
Using the definition of $g_{\lambda}$, it is not difficult to see that 
$$
\mathrm{P}_{g,b,c}(\lambda) =\sum_{\mathrm{sh}(\pi) = \lambda} \mathrm{P}_{b,c}(\pi) =\frac{1}{Z_{b,c}} \sum_{\mathrm{sh}(\pi) = \lambda} \prod_{(i,j) \in \mathrm{Des}(\pi)} q_{\pi_{i j}} = \frac{1}{Z_{b,c}} g_{\lambda}(q_1, \ldots, q_c),
$$
where the sum runs over $\pi \in \mathrm{PP}(\infty, b, c)$. 
\end{proof}

\begin{corollary}\label{znorm}
We have 
$$
Z_{b,c} = \prod_{i = 1}^{c} (1 - q_i)^{-b}.
$$
\end{corollary}

\begin{proof}
Follows from the proposition and the identity \eqref{gl}.
\end{proof}

\begin{theorem}\label{fpd} 
For $\pi \in \mathrm{PP}(\infty, b, c)$, let $\lambda = (\lambda_1 \ge \cdots \ge \lambda_b \ge 0)$ be the shape of $\pi$. 
The distribution of the first part $\lambda_1$ can be 
expressed as 
\begin{align}\label{toep11}
\mathrm{P}_{g,b,c}(\lambda_1 \le a) 
= \frac{1}{Z_{b,c}} s_{(a^b)}(1^{b}, q_1, \ldots, q_c). 
\end{align}
\end{theorem}
\begin{proof}
For the distribution of $\lambda_1$ we now have
\begin{align*}
\mathrm{P}_{g,b,c}(\lambda_1 \le a) &= \sum_{\lambda \subset (a^b)} \mathrm{P}_{g,b,c}(\lambda) \\
&= \frac{1}{Z_{b,c}} \sum_{\lambda \subset (a^b)} g_{\lambda}(q_1, \ldots, q_c)\\ 
&= \frac{1}{Z_{b,c}} g_{(a^b)}(1,q_1, \ldots, q_c)\\ 
&= \frac{1}{Z_{b,c}} s_{(a^b)}(1^b,q_1, \ldots, q_c),
\end{align*}
where we used  Lemma~\ref{l1} in the last identity.
\end{proof}

\begin{corollary} 
The distribution of $\lambda_1$ can be 
expressed via the Toeplitz determinant
\begin{align}\label{toep}
\mathrm{P}_{b,c}(\lambda_1 \le a) 
=  \frac{1}{Z_{b,c}}\det\left[\hat\varphi_{i - j}\right]_{i,j=1}^{a},
\end{align}
where 
$$\hat\varphi_{k} = \sum_{\ell = 0}^{\infty} \binom{b}{\ell + k} 
e_{\ell}(q_1, \ldots, q_c)\quad \text{ or }\quad
 \varphi(z) := \sum_{k \in \mathbb{Z}} \hat\varphi_k z^k = (1 + z)^b \prod_{i = 1}^c (1 + q_i z^{-1}).$$
\end{corollary}

\begin{proof}
Using the Jacobi-Trudi identity for Schur polynomials we obtain 
$$
s_{(a^b)}(1^b, q_1, \ldots, q_c) = \det[e_{b - i + j}(1^b, q_1, \ldots, q_c)]_{i,j = 1}^{a}
$$
Note that 
$$
e_{b - i + j}(1^b, q_1, \ldots, q_c) = \sum_{\ell \ge 0} \binom{b}{b - i + j - \ell} e_{\ell}(q_1, \ldots, q_c) = \sum_{\ell \ge 0} \binom{b}{\ell + i - j} e_{\ell}(q_1, \ldots, q_c) = \hat\varphi_{i - j}
$$
as needed.
\end{proof}

\begin{remark}
From the Toeplitz determinant we can also obtain {\it Fredholm determinant} formula of the form $\det[I - K]$ for some kernel $K$ using the Borodin-Okounkov formula \cite{bo}.
\end{remark}

\begin{remark}
Let $\Lambda$ be the ring of symmetric functions. A specialization (ring homomorphism) $\rho : \Lambda \to \mathbb{R}$ 
is called {\it $g$-positive} if $\rho(g_{\lambda}) \ge 0$ for all $\lambda$. In light of the $g$-measure one can be interested in characterizations of $g$-positive specializations. See \cite{dy3} for a class of $g$-positive specializations and related questions. 
\end{remark}

\section{Corner growth model}\label{corner}
Consider a 
random matrix $W = (w_{ij})$ with i.i.d. entries $w_{ij}$ that are geometrically distributed with parameter $q$, 
i.e. $$\mathrm{P}(W_{ij} = k) = (1 - q) q^k, \quad k \in \mathbb{N}.$$ 

For $m,n \ge 1$, recall that the {last passage times} $G(m,n)$ are defined as
$$
G(m,n) = \max_{\Pi} \sum_{(i,j) \in \Pi} w_{i j},
$$
where the maximum is taken over monotone paths $\Pi$ from (1,1) to $(m,n)$. Note that we have the recurrence relation
$$
G(i,j) = \max(G(i-1, j), G(i,j-1)) + w_{ij}, \quad i, j \ge 1
$$
(one needs to set the boundary conditions $G(i,0) = G(0,i) = 0$).

The {\it corner growth model} can be viewed as evolution of a random Young diagram $Y(t)$ given at time $t$ by the region
$$Y(t) = \{(i,j) \in \mathbb{N}^2 : G(i,j) \le t \}.$$
One can see that $Y(t+1)$ adds boxes to some corners of $Y(t)$. 
We refer to \cite{sep, romik} for more on the topic. 

The main result of this section is the following.

\begin{theorem}\label{main}
Let $\lambda = (\lambda_1, \ldots, \lambda_b)$ be a partition with at most $b$ parts. 
We have
\begin{align*}
\mathrm{P}(G({b,c}) = \lambda_1, \ldots, G(1,c) = \lambda_b) = \mathrm{P}_{g,b,c}(\lambda) = (1-q)^{bc} g_{\lambda}(\underbrace{q, \ldots, q}_{c \text{ times}}), %
\end{align*}
where $q_i = q$ for all $i \in [1,c]$ in the $g$-measure.
\end{theorem}

To prove this theorem we need the following lemma. 

\begin{lemma}\label{gen}
Let $D = (D_{i j})_{i,j = 1}^{b,c}$ be a random $b \times c$ matrix, where $D_{i j}$ are i.i.d. geometrically distributed random variables with parameter $q$.  
Then 
$$\mathrm{P}(\Phi^{-1}(D) = \pi)= \mathrm{P}_{g,b,c}(\pi),$$
where $q_i = q$ for all $i \in [1,c]$ in the $g$-measure.
\end{lemma}
\begin{proof}
Note firstly that for $q_i = q$ we have
$$
\mathrm{P}_{g,b,c}(\pi) = 
(1 - q)^{bc} q^{\mathrm{des}(\pi)}. 
$$
Let $d_{i \ell} = |\{j: \pi_{ij} = \ell > \pi_{i+1 j} \}|$, i.e. $(d_{i \ell})= \Phi(\pi)$. Then 
$$
\mathrm{P}(D= \Phi(\pi)) = \prod_{i = 1}^{b} \prod_{\ell = 1}^{c} \mathrm{P}(D_{i \ell} = d_{i \ell}) = \prod_{i = 1}^b \prod_{\ell = 1}^c (1 - q)\, q_{}^{d_{i \ell}} = \mathrm{P}_{b,c}(\pi)
$$
as needed.
\end{proof}

\begin{proof}[Proof of Theorem~\ref{main}]
Follows now from Lemma~\ref{phip} (ii), Lemma~\ref{gen} and Proposition~\ref{gme}.
\end{proof}

We now list some implications of Theorem~\ref{main}. 
\begin{corollary}[Marginal distributions]\label{marg}
Let $k \in [1,b]$. We have
$$
\mathrm{P}(G(b-k +1,c) \le n) = \mathrm{P}_{g,b,c}(\lambda_k \le n).
$$
\end{corollary}

In particular, $\lambda_b \sim G(1,c)$ has binomial distribution with parameters $c$, $q$. For $\lambda_1 \sim G(b,c)$  
combining with Theorem~\ref{fpd} we obtain.
\begin{corollary}
We have 
$$
\mathrm{P}(G(b,c) \le a) = (1-q)^{bc} s_{(a^b)}(1^b, \underbrace{q, \ldots, q}_{c \text{ times}})
$$
as well as Toeplitz determinantal expressions \eqref{toep}. 
\end{corollary}

More generally using the Jacobi-Trudi-type determinantal formula \eqref{jt} we get
\begin{corollary}
We have the following formula
$$
\mathrm{P}(G(b,c) = \lambda_1,\ldots, G(1,c) = \lambda_b) = (1-q)^{bc} \det\left[e_{\lambda'_i - i + j}(1^{\lambda'_i - 1}, \underbrace{q, \ldots, q}_{c \text{ times}}) \right]_{i,j = 1}^{\lambda_1}
$$
\end{corollary}

Let us now discuss several results about the corner growth model that imply results on the shape $\lambda$ from the $g$-distribution. 

It was proved in \cite{joh1} that for $x > 0$ 
we have almost sure convergence to deterministic {limit shape}
\begin{align*}
\lim_{b \to \infty} \frac{1}{b}{G(\lfloor x b \rfloor, b)} \to \psi(x) = (1- q)^{-1} (qx + q + 2\sqrt{qx})
\end{align*}
and {\it fluctuations} around this limit shape are of order $b^{1/3}$ given by  
\begin{align*}
\lim_{b \to \infty} \mathrm{P}\left(\frac{G(\lfloor x b \rfloor, b) - \psi(x) b}{\sigma(x) b^{1/3}} \le t \right) = F_{\mathrm{TW}}(t),
\end{align*} 
where $\sigma(x) = (1-q)^{-1}(qx)^{1/6}(\sqrt{x} + \sqrt{q})^{2/3} (1 + \sqrt{qx})^{2/3}$ and $F_{\mathrm{TW}}(t)$
is the {\it Tracy-Widom distribution}, the limiting distribution of the properly scaled largest eigenvalue in GUE \cite{tw}. Therefore, combining these with Corollary~\ref{marg}, $\lambda$ from the $g$-distribution has limit shape
$$
\lim_{b \to \infty} \frac{\lambda_{\lfloor x b\rfloor}}{b} \to \psi(1 - x) 
$$
where $x \in (0,1]$, $c = b$; and similar fluctuations.

It was proved in \cite{bar} that for fixed $b$, as $c \to \infty$, the variables 
\begin{align}\label{gc}
\frac{G(k,c) - \frac{q}{1-q} c}{\frac{\sqrt{q}}{1-q} \sqrt{c}}, \quad k = 1, \ldots, b
\end{align}
jointly converge in distribution to largest eigenvalues of $b$ principal minors in $b \times b$ GUE matrix, which is known as {\it GUE corners process} (or GUE minors process \cite{jn}). 
Therefore, for $\lambda$  from the $g$-distribution we have
\begin{align}\label{gc1}
\frac{\lambda_{b - k + 1} - \frac{q}{1-q} c}{\frac{\sqrt{q}}{1-q} \sqrt{c}}, \quad k = 1, \ldots, b
\end{align}
jointly converge in distribution to largest eigenvalues in $b \times b$ GUE corners process. It is also known \cite{gw} that as $c \to \infty$ the 
random variables \eqref{gc}
converge in distribution to the process
$$
\sup_{0 = t_0 < \cdots < t_k = 1} \sum_{i = 0}^{k - 1} (B_i(t_{i + 1}) - B_i(t_i)), \quad k = 1,2,\ldots
$$
where $B_i$ are independent standard Brownian motions. Hence the same holds for the random variables \eqref{gc1}. 

\begin{remark}\label{upd}
After this paper was written, more related works on probabilistic models for dual Grothendieck polynomials appeared in \cite{dy5, ms20}.
\end{remark}
\begin{remark}
The $g$-measure and Theorem~\ref{main} have higher-dimensional generalizations 
obtained recently in \cite{ay}.
\end{remark}
 
\section*{Acknowledgements}
I am grateful to Pavlo Pylyavskyy for many stimulating conversations on $K$-theoretic combinatorics. I am also grateful to Leonid Petrov for helpful comments and some references. I~am also thankful to the referees for careful reading of the text and many useful comments. 


\end{document}